\documentclass[12pt,draft]{amsart}
\usepackage{setspace,units}

\newtheorem{theorem}{Theorem}[section]
\newtheorem{lemma}[theorem]{Lemma}

\theoremstyle{definition}

\newtheorem{remark}[theorem]{Remark}
\newtheorem{Conjecture}[theorem]{Conjecture}
\newtheorem{Question}[theorem]{Question}

\numberwithin{equation}{section}

\RequirePackage{amsmath}
\RequirePackage{amssymb}
\RequirePackage{amsthm}
 \RequirePackage{epsfig}

\begin{document}

\date{}
\title[Zero-one laws in Diophantine approximation]
{Zero-one laws in simultaneous and multiplicative Diophantine
approximation}

\author{ Liangpan Li}

\address{Department of Mathematical Sciences,
Loughborough University, LE11 3TU, UK}
 \email{liliangpan@gmail.com}



\date{}

\begin{abstract}
Answering two questions of Beresnevich and Velani, we develop
zero-one laws in both simultaneous and multiplicative Diophantine
approximation. Our proofs rely on a Cassels-Gallagher type theorem
as well as a higher-dimensional analogue of the cross fibering
principle of Beresnevich, Haynes and Velani.
\end{abstract}

\maketitle


\section{Introduction}

 Diophantine approximation is
the quantitative study of
 rational number approximation to real numbers, originating from the celebrated theorem of Dirichlet that for any
irrational number $\alpha$, there exist infinitely many
$(q,p)\in\mathbb{N}\times\mathbb{Z}$ satisfying
$|\alpha-\frac{p}{q}|<\frac{1}{q^2}$. Note also for each algebraic
number $\xi$ of degree $d\geq2$, an inequality of Liouville produces
a constant $c(\xi)>0$ such that for all
$(q,p)\in\mathbb{N}\times\mathbb{Z}$,
$|\xi-\frac{p}{q}|>\frac{c(\xi)}{q^d}$, and the Thue-Siegel-Roth
theorem (\cite{Roth}) further gives for any real number $\gamma>2$,
a constant $c=c(\xi,\gamma)>0$ such that for all
$(q,p)\in\mathbb{N}\times\mathbb{Z}$,
$|\xi-\frac{p}{q}|>\frac{c}{q^{\gamma}}$.

There is in flavour another type of questions in Diophantine
approximation, that is, to prove statements which are `almost
always' or `almost never' true. For example, given
$\psi:\mathbb{N}\rightarrow[0,\frac{1}{2})$ often referred to as an
approximation function, a long-standing conjecture of Duffin and
Schaeffer (\cite{DuffinSchaeffer}) says that
\[{\mathcal M}_1\big(\limsup_{q\rightarrow\infty}{\mathcal E}_q(\psi)\big)=1
\Longleftrightarrow \sum_{q\in\mathbb{N}}{\mathcal M}_1({\mathcal
E}_q(\psi))=\infty,\] where  ${\mathcal M}_{1}(\cdot)$ denotes the
one-dimensional Lebesgue measure,
\[{\mathcal E}_q(\psi)\triangleq\bigcup_{p=1 \atop
(p,q)=1}^q\big(\frac{p-\psi(q)}{q},\frac{p+\psi(q)}{q}\big).\]In
this paper we will study a specific question of such flavour, that
is, the so-called zero-one law in
 Diophantine approximation. Given an approximation function  $\psi$, Cassels (\cite{Cassels})
proved that the measure of the set of real numbers $x$ in
$\mathbb{I}\triangleq[0,1]$ for which $|qx-p|<\psi(q)$ holds for
infinitely many pairs of positive integers $(q,p)$ equals either 0
or 1, while Gallagher (\cite{Gallagher}) showed ${\mathcal
M}_1(\limsup_{q\rightarrow\infty}{\mathcal E}_q(\psi))\in\{0,1\}$.
The Cassels-Gallagher theorems called also zero-one laws have played
a fundamental role in the study of metric Diophantine approximation
(see e.g. \cite{Harman,Sprindzuk}).

One may naturally consider developing zero-one laws in simultaneous
 Diophantine approximation. Given
$\{\Psi_j:\mathbb{Z}^n\backslash\{\mathbf{0}\}\rightarrow\mathbb{R}^{+}\}_{j=1}^m$,
let ${\mathcal A}_{n,m}(\Psi_1,\ldots,\Psi_m)$, ${\mathcal
B}_{n,m}(\Psi_1,\ldots,$ $\Psi_m)$, ${\mathcal
C}_{n,m}(\Psi_1,\ldots,\Psi_m)$ denote the sets of
$(\mathbf{X}^{(1)},\ldots,$ $\mathbf{X}^{(m)})\in\mathbb{I}^{nm}$
for which
\begin{align}\label{formula 11}
|\mathbf{q}\mathbf{X}^{(j)}+p_j|<\Psi_j(\mathbf{q})\ \ \
(j=1,\ldots,m)
\end{align} holds for infinitely many pairs $\mathbf{q}\in
\mathbb{Z}^n\backslash\{\mathbf{0}\}$ and
$\mathbf{p}=(p_1,\ldots,p_m)\in\mathbb{Z}^m$ subject to respectively
1) free condition on $\mathbf{q}$ and $\mathbf{p}$; 2) local
coprimality condition on $\mathbf{q}$ and $p_j$ for each $j$; 3)
global coprimality condition on $\mathbf{q}$ and $\mathbf{p}$, where
$\mathbf{q}$ is regarded as a row while $\mathbf{X}^{(j)}$ a column,
two integer lattices (may be of different dimensions) are said to be
coprime if the greatest common divisor of all of their components is
1.  In the case $\Psi_1=\cdots=\Psi_m$, Beresnevich and Velani
(\cite[Theorem 1]{BV1}; see also \cite{Vilchinski}) proved that
${\mathcal M}_{nm}({\mathcal
F}_{n,m}(\Psi_1,\ldots,\Psi_1))\in\{0,1\}$, here and later on
$\mathcal F$ stands for any of $\mathcal A$, $\mathcal B$, $\mathcal
C$. They (\cite[Theorem 4]{BV1}) also showed in general
\begin{align}
\label{formula 12} {\mathcal
M}_{nm}\big(\bigcup_{k\in\mathbb{N}}{\mathcal
F}_{n,m}(k\Psi_1,\ldots,k\Psi_m)\big)\in\{0,1\},
\end{align}
and asked  among several others (\cite[Question 2]{BV1}) to prove
that
\begin{align}\label{formula 13}
{\mathcal M}_{nm}\big(\bigcup_{k\in\mathbb{N}}{\mathcal
F}_{n,m}(k\Psi_1,\ldots,k\Psi_m)\big)={\mathcal
M}_{nm}\big(\bigcap_{k\in\mathbb{N}}{\mathcal
F}_{n,m}(\frac{\Psi_1}{k},\ldots,\frac{\Psi_m}{k})\big).
\end{align}

We may also study zero-one laws in multiplicative Diophantine
approximation in a similar way. Given
$\Psi:\mathbb{Z}^n\backslash\{\mathbf{0}\}\rightarrow\mathbb{R}^{+}$,
let ${\mathcal A}_{n,m}^{\times}(\Psi)$, ${\mathcal
B}_{n,m}^{\times}(\Psi)$, ${\mathcal C}_{n,m}^{\times}(\Psi)$ denote
the sets of $(\mathbf{X}^{(1)},\ldots,$
$\mathbf{X}^{(m)})\in\mathbb{I}^{nm}$ for which
\begin{align}\label{formula 14}
\prod_{j=1}^m|\mathbf{q}\mathbf{X}^{(j)}+p_j|<\Psi(\mathbf{q})
\end{align} holds for infinitely many pairs $\mathbf{q}\in
\mathbb{Z}^n\backslash\{\mathbf{0}\}$ and
$\mathbf{p}=(p_1,\ldots,p_m)\in\mathbb{Z}^m$ subject to respectively
1) free condition on $\mathbf{q}$ and $\mathbf{p}$; 2) local
coprimality condition on $\mathbf{q}$ and $p_j$ for each $j$; 3)
global coprimality condition on $\mathbf{q}$ and $\mathbf{p}$. In
the case ${\mathcal G}\in\{{\mathcal A},{\mathcal B}\}$,
Beresnevich, Haynes and Velani (\cite[Theorem 1]{BHV}) obtained
 ${\mathcal M}_{nm}({\mathcal
G}_{n,m}^{\times}(\Psi))\in\{0,1\}$. Beresnevich and Velani
(\cite[Question 5]{BV1}) also asked to prove that
\begin{align}\label{formula 145}{\mathcal
M}_{nm}\big({\mathcal
A}_{n,m}^{\times}(\Psi)\backslash\bigcap_{k\in\mathbb{N}}{\mathcal
A}_{n,m}^{\times}(\frac{\Psi}{k})\big)=0.
\end{align}

 The main purpose of this paper is to show that
\begin{align}
\label{formula 15}  {\mathcal M}_{nm}({\mathcal
F}_{n,m}(k_1\Psi_1,\ldots,k_m\Psi_m))\in\{0,1\}\ \mbox{is
independent of}\ \{k_j>0\}_{j=1}^m, \end{align}
and
\begin{align}\label{formula 16} {\mathcal M}_{nm}({\mathcal
F}_{n,m}^{\times}(k\Psi))\in\{0,1\}\ \mbox{is independent of}\ k>0,
\end{align}
which result in (\ref{formula 13}), (\ref{formula 145}), and all
aforementioned zero-one laws as special cases.

We will prepare several preliminary facts
in Section
\ref{preliminaries}, and study zero-one laws in simultaneous and
multiplicative Diophantine approximation respectively in Section
\ref{n=1} and Section \ref{section 4}. A few remarks will be
addressed in the last section.

\section{Preliminaries}\label{preliminaries}

\subsection{Cassels-Gallagher theorem}

For each $\mathbf{q}\in\mathbb{Z}^n\backslash\{\mathbf{0}\}$, let
$\omega(\mathbf{q})$ be a fixed subset of divisors of
$\mbox{gcd}(\mathbf{q})$, and let
$M_{\mathbf{q}}:\omega(\mathbf{q})\rightarrow\mathbb{R}^{+}$ be a
fixed function. Let ${\mathcal H}_n(\omega,M)$ denote the set
\[\big\{\mathbf{X}\in\mathbb{I}^n:|\mathbf{q}\mathbf{X}+p|<
M_{\mathbf{q}}(\gamma)\ \mbox{for i.m.}\ (\mathbf{q},p)\
\mbox{with}\ p\in\mathbb{Z}\ \mbox{coprime to}\ \mbox{some}\
\gamma\in\omega({\mathbf{q}})\big\},\]
 where `i.m.'
means `infinitely many'.

\begin{lemma}\label{lemma 21} ${\mathcal M}_n({\mathcal
H}_n(\omega,kM))\in\{0,1\}$ is independent of $k>0$.
\end{lemma}

The proof of Lemma \ref{lemma 21} needs the following two lemmas,
whose proofs and one-dimensional prototypes can be found in
\cite{BV1,Cassels,Gallagher}.

\begin{lemma}\label{lemma 22}
Let $\{B_i\}_{i\in\mathbb{N}}$ be a sequence of balls in
$\mathbb{R}^n$ with ${\mathcal M}_n(B_i)\rightarrow0$ as
$i\rightarrow\infty$. Let $\{U_i\}_{i\in\mathbb{N}}$ be a sequence
of measurable sets such that $U_i\subset B_i$ for all
$i\in\mathbb{N}$. Assume that for some constant $c>0$, ${\mathcal
M}_n(U_i)\geq c{\mathcal M}_n(B_i)$ for all $i\in\mathbb{N}$. Then
the upper limits of $\{B_i\}_{i\in\mathbb{N}}$ and
$\{U_i\}_{i\in\mathbb{N}}$ have the same measure.
\end{lemma}

\begin{lemma}\label{lemma 23}
For any integer $s\geq2$ and $\mathbf{e}\in\mathbb{Z}^n$ consider
the transformation of $\mathbb{I}^n$ into itself given by
\[T:\mathbf{X}\mapsto s\mathbf{X}+\frac{\mathbf{e}}{s}\ \ \ (\mbox{mod}\ 1).\]
Suppose $A\subset\mathbb{I}^n$ with $T(A)\subset A$. Then ${\mathcal
M}_n(A)=0$ or 1.
\end{lemma}

\begin{proof}[Proof of the independence part of Lemma \ref{lemma
21}] By letting $|\cdot|$ denote the supremum norm in
$\mathbb{R}^n$, we have two cases to consider.

\textsc{Case 1}: Suppose $\limsup_{|\mathbf{q}|\rightarrow\infty}
\frac{\displaystyle\max_{\gamma\in\omega(\mathbf{q})}M_{\mathbf{q}}(\gamma)}{\displaystyle|\mathbf{q}|}>0$.
In this case there exist a $\delta>0$ and
$\{\gamma^{(i)}\in\omega(\mathbf{q}^{(i)})\}_{i\in\mathbb{N}}$ with
$|\mathbf{q}^{(i)}|\rightarrow\infty$ as $i\rightarrow\infty$ such
that
\[\frac{\displaystyle M_{\mathbf{q}^{(i)}}(\gamma^{(i)})}{\displaystyle|\mathbf{q}^{(i)}|}\geq\delta\ \ \ (\forall i\in\mathbb{N}).\]
Write $\mathbf{q}^{(i)}=(q^{(i)}_1,\ldots,q^{(i)}_n)$ and assume
without loss of generality that $|\mathbf{q}^{(i)}|=|{q}_1^{(i)}|$
for all $i\in\mathbb{N}$. From Gallagher's proof (\cite[Lemma
1]{Gallagher}) we know that if $|q|$ is large enough, then for any
$\alpha\in\mathbb{R}$,
\[\{x_1\in\mathbb{I}:|x_1+\frac{\alpha+p}{q}|<k\delta\ \mbox{for some}\
p\in\mathbb{Z}\ \mbox{with}\ (p,q)=1\}=\mathbb{I}.\] Thus for any
$x_2,\ldots,x_n\in\mathbb{I}$,
\[\{x_1\in\mathbb{I}:\big|x_1+\frac{\sum_{j=2}^nq_j^{(i)}x_j+p}{\displaystyle q_1^{(i)}}\big|
<\frac{\displaystyle k
M_{\mathbf{q}^{(i)}}(\gamma^{(i)})}{\displaystyle |{q}_1^{(i)}|}\
\mbox{for some}\ p\in\mathbb{Z}\ \mbox{with}\
(p,\gamma^{(i)})=1\}=\mathbb{I}\]  provided $|q_1^{(i)}|$ is large
enough. By definition, ${\mathcal H}_n(\omega,kM)=\mathbb{I}^n$.

\textsc{Case 2}: Suppose $\lim_{|\mathbf{q}|\rightarrow\infty}
\frac{\displaystyle\max_{\gamma\in\omega(\mathbf{q})}M_{\mathbf{q}}(\gamma)}{\displaystyle|\mathbf{q}|}=0$.
For any triple
$(\mathbf{q},p,\gamma)\in(\mathbb{Z}^n\backslash\{\mathbf{0}\})
\times\mathbb{Z}\times(\mathbb{Z}\backslash\{0\})$ with
$N(\mathbf{q},p)\triangleq\{\mathbf{X}\in\mathbb{I}^n:\mathbf{q}\mathbf{X}+p=0\}$
being non-empty, $\gamma\in\omega(\mathbf{q})$, $(\gamma,p)=1$ and
$M_\mathbf{q}(\gamma)>0$, we fix a finite subset
$N(\mathbf{q},p,\gamma)$ of $N(\mathbf{q},p)$ such that for any
$\mathbf{X}\in N(\mathbf{q},p)$ there is a $\mathbf{Y}\in
N(\mathbf{q},p,\gamma)$ satisfying
$|\mathbf{X}-\mathbf{Y}|<\frac{M_{\mathbf{q}}(\gamma)}{|\mathbf{q}|}$.
All other triples $(\mathbf{q},p,\gamma)$ will be completely ignored
since they are useless to our study. For any $k>0$ we define
\[{\mathcal H}_{\mathbf{q},p,\gamma}(k)=\bigcup_{\mathbf{Y}\in N(\mathbf{q},p,\gamma)}B(\mathbf{Y},k\frac{M_{\mathbf{q}}(\gamma)}{|\mathbf{q}|})\]
and
\[{\mathcal H}_{\mathbf{q}}(k)=\bigcup_{p}\bigcup_{\gamma}{\mathcal H}_{\mathbf{q},p,\gamma}(k),\]
where $B(\mathbf{Y},r)$ denotes the ball in $\mathbb{R}^n$ with
center $\mathbf{Y}$ and radius $r$.  Then
$\forall\epsilon\in(0,\frac{1}{2})$,  it is easy to verify that for
any sufficient large $|\mathbf{q}|$,
\[[\epsilon,1-\epsilon]^n\cap {\mathcal H}_{\mathbf{q},p,\gamma}(\frac{k}{n})\subset\{\mathbf{X}\in\mathbb{I}^n:|\mathbf{q}\mathbf{X}+p|<
kM_{\mathbf{q}}(\gamma)\},\]
\[[\epsilon,1-\epsilon]^n\cap \{\mathbf{X}\in\mathbb{I}^n:|\mathbf{q}\mathbf{X}+p|<
kM_{\mathbf{q}}(\gamma)\}\subset {\mathcal
H}_{\mathbf{q},p,\gamma}(k+\sqrt{n}).\] These facts naturally imply
that
\[{\mathcal M}_n(\limsup_{|\mathbf{q}|\rightarrow\infty}{\mathcal
H}_{\mathbf{q}}(\frac{k}{n}))\leq{\mathcal M}_n({\mathcal
H}_n(\omega,kM))\leq{\mathcal
M}_n(\limsup_{|\mathbf{q}|\rightarrow\infty}{\mathcal
H}_{\mathbf{q}}(k+\sqrt{n})),\] and consequently by Lemma \ref{lemma
22}, ${\mathcal M}_n({\mathcal H}_n(\omega,kM))$ is independent of
$k>0$.
\end{proof}

\begin{proof}[Proof of the zero-one part of Lemma \ref{lemma 21}]
For each prime number $s$ and each non-negative integer $\nu$, we
consider the approximation
\begin{align}\label{formula 21}
|\mathbf{q}\mathbf{X}+p|<
 s^{\nu}M_{\mathbf{q}}(\gamma)\ \ \ \mbox{for some coprime pair}\ \gamma\in\omega(\mathbf{q})\ \mbox{and}\
 p\in\mathbb{Z}
 \end{align}
and define three non-decreasing sequences of sets ${\mathcal
R}(s^{\nu})$, ${\mathcal S}(s^{\nu})$ and  ${\mathcal T}(s^{\nu})$
as follows:
\begin{align*}
&\mathbf{X}\in{\mathcal R}(s^{\nu})\ \mbox{if}\ \mathbf{X}\
\mbox{satisfies}\ (\ref{formula 21})\ \mbox{for i.m.}\ \mathbf{q}\
\mbox{with}\ s\nmid \mbox{gcd}(\mathbf{q}),  \\
&\mathbf{X}\in{\mathcal S}(s^{\nu})\, \, \,  \mbox{if}\ \mathbf{X}\
\mbox{satisfies}\ (\ref{formula 21})\ \mbox{for i.m.}\ \mathbf{q}\
\mbox{with}\ s|\mbox{gcd}(\mathbf{q})\ \&\ s^2\nmid\mbox{gcd}(\mathbf{q}),\\
&\mathbf{X}\in{\mathcal T}(s^{\nu})\ \mbox{if}\ \mathbf{X}\
\mbox{satisfies}\ (\ref{formula 21})\ \mbox{for i.m.}\ \mathbf{q}\
\mbox{with}\ s^2|\mbox{gcd}(\mathbf{q}).
\end{align*}

If $\mathbf{X}$ satisfies (\ref{formula 21}) with $s\nmid
\mbox{gcd}(\mathbf{q})$, then
\[|\mathbf{q}\langle s\mathbf{X}\rangle+\mathbf{q}\lfloor s\mathbf{X}\rfloor+sp|<s^{\nu+1}M_{\mathbf{q}}(\gamma)\ \
(\gamma,\mathbf{q}\lfloor s\mathbf{X}\rfloor+sp)=1,\] where as usual
$\langle(x_1,\ldots,x_n)\rangle\triangleq(\langle
x_1\rangle,\ldots,\langle x_n\rangle)$,
$\lfloor(x_1,\ldots,x_n)\rfloor\triangleq(\lfloor
x_1\rfloor,\ldots,\lfloor x_n\rfloor)$, $\langle\alpha\rangle$ and
$\lfloor\alpha\rfloor$ denote respectively the fractional and
integer parts of $\alpha\in\mathbb{R}$. (Reason: We argue by
contradiction and suppose there exists a prime number $z$ such that
$z|\gamma$, $z|\mathbf{q}\lfloor s\mathbf{X}\rfloor+sp$. Since
$\gamma$ is a divisor of $\mbox{gcd}(\mathbf{q})$, $z|sp$, which by
the primality of $s,z$ gives either $z=s$ or $z|p$. In the first
case, we have $s=z|\gamma|\mbox{gcd}(\mathbf{q})$, a contradiction.
In the second one we have $z|(\gamma,p)=1$, also a contradiction.)
This shows the map $\mathbf{X}\mapsto\langle s\mathbf{X}\rangle$
sends $\cup_{\nu\geq0}{\mathcal R}(s^{\nu})$ into itself. By Lemma
\ref{lemma 23}, ${\mathcal M}_n(\cup_{\nu\geq0}{\mathcal
R}(s^{\nu}))\in\{0,1\}$.

If $\mathbf{X}$ satisfies (\ref{formula 21}) with $s|
\mbox{gcd}(\mathbf{q})$ and $s^2\nmid q_i$ for some
$i\in\{1,\ldots,n\}$, then
\[|\mathbf{q}\langle s\mathbf{X}+\frac{\mathbf{1}_i}{s}\rangle+\mathbf{q}\lfloor s\mathbf{X}+\frac{\mathbf{1}_i}{s}\rfloor
-\frac{q_i}{s}+sp| <s^{\nu+1}M_{\mathbf{q}}(\gamma)\ \
 (\gamma,\mathbf{q}\lfloor s\mathbf{X}+\frac{\mathbf{1}_i}{s}\rfloor-\frac{q_i}{s}+sp)=1,\]
 where $\mathbf{1}_i$ denotes the element of $\mathbb{Z}^n$ with zero entries everywhere except in the
 $i$-th position where the entry is 1.
(Reason: We argue by contradiction and suppose there exists a prime
number $z$ such that $z|\gamma$, $z|\mathbf{q}\lfloor
s\mathbf{X}+\frac{\mathbf{1}_i}{s}\rfloor-\frac{q_i}{s}+sp$. Since
$\gamma$ is a divisor of $\mbox{gcd}(\mathbf{q})$,
$z|sp-\frac{q_i}{s}$. Note $z|s\frac{q_i}{s}$, which by the
primality of $s,z$ gives either $z=s$ or $z|\frac{q_i}{s}$. In the
first case, we have $s|\frac{q_i}{s}$, a contradiction. In the
second case we have $z|sp$. Now there are two subcases to consider,
one is $z=s$, the other is $z|p$. If the first subcase happens, then
$s|\frac{q_i}{s}$, a contradiction; if the second one happens, then
$z|(\gamma,p)=1$, also a contradiction.) This shows  the map
$\mathbf{X}\mapsto\langle s\mathbf{X}+\frac{\mathbf{1}_i}{s}\rangle$
sends $\cup_{\nu\geq0}{\mathcal S}^{(i)}(s^{\nu})$ into itself,
where
\[\mathbf{X}\in{\mathcal S}^{(i)}(s^{\nu})\ \mbox{if}\ \mathbf{X}\
\mbox{satisfies}\ (\ref{formula 21})\ \mbox{for i.m.}\ \mathbf{q}\
\mbox{with}\ s|\mbox{gcd}(\mathbf{q})\ \&\ s^2\nmid q_i.\] By Lemma
\ref{lemma 23}, ${\mathcal M}_n(\cup_{\nu\geq0}{\mathcal
S}^{(i)}(s^{\nu}))\in\{0,1\}$. Note  ${\mathcal
S}(s^{\nu})=\cup_{i=1}^n{\mathcal S}^{(i)}(s^{\nu})$, which easily
implies that ${\mathcal M}_n(\cup_{\nu\geq0}{\mathcal
S}(s^{\nu}))\in\{0,1\}$.

If $\mathbf{X}$ satisfies (\ref{formula 21}) with
$s^2|\mbox{gcd}(\mathbf{q})$, then for each $i\in\{1,\ldots,n\}$,
\[|\mathbf{q}\langle \mathbf{X}+\frac{\mathbf{1}_i}{s}\rangle+\mathbf{q}\lfloor \mathbf{X}+\frac{\mathbf{1}_i}{s}\rfloor-\frac{q_i}{s}+p|
<s^{\nu}M_{\mathbf{q}}(\gamma)\ \
 (\gamma,\mathbf{q}\lfloor \mathbf{X}+\frac{\mathbf{1}_i}{s}\rfloor-\frac{q_i}{s}+p)=1.\]
(Reason: We argue by contradiction and suppose there exists a prime
number $z$ such that $z|\gamma$, $z|\mathbf{q}\lfloor
\mathbf{X}+\frac{\mathbf{1}_i}{s}\rfloor-\frac{q_i}{s}+p$. Since
$\gamma$ is a divisor of $\mbox{gcd}(\mathbf{q})$,
$z|p-\frac{q_i}{s}$. Note $z|s\frac{q_i}{s}$, which by the primality
of $s,z$ gives either $z=s$ or $z|\frac{q_i}{s}$. In the first case
since $s|\frac{q_i}{s}$ we have $z|p$, while in the second one we
can also have $z|p$. So no matter which case happens, $z|p$, and
consequently, $z|(\gamma,p)=1$, a contradiction.) This shows
 ${\mathcal T}(s^{\nu})$ has period
$\frac{1}{s}$ in each variable, so is $\cup_{\nu\geq0}{\mathcal
T}(s^{\nu})$.

If either ${\mathcal M}_n(\cup_{\nu\geq0}{\mathcal R}(s^{\nu}))=1$
or ${\mathcal M}_n(\cup_{\nu\geq0}{\mathcal S}(s^{\nu}))=1$, then it
is easy to see that ${\mathcal M}_n(\cup_{k\in\mathbb{N}}{\mathcal
H}_n(\omega,kM))=1$. Else we assume ${\mathcal
M}_n(\cup_{\nu\geq0}{\mathcal R}(s^{\nu}))={\mathcal
M}_n(\cup_{\nu\geq0}{\mathcal S}(s^{\nu}))=0$. In this case we
observe that for any prime number $s$,
$\cup_{k\in\mathbb{N}}{\mathcal H}_n(\omega,kM)$ differs from
$\cup_{\nu\geq0}{\mathcal T}(s^{\nu})$ by a null set. Roughly
speaking, $\cup_{k\in\mathbb{N}}{\mathcal H}_n(\omega,kM)$ is a
periodic set of sufficiently small period uniformly in every
variable. Thus a standard application of the Lebesgue density
theorem gives ${\mathcal M}_n(\cup_{k\in\mathbb{N}}{\mathcal
H}_n(\omega,kM))\in\{0,1\}$ (see e.g. \cite{BV1,Gallagher}). This
finishes the whole proof.
\end{proof}

\begin{remark}
Lemma \ref{lemma 21} reduces to the Cassels theorem (\cite[Theorem
VIII]{Cassels}) and the Gallagher theorem (\cite[Theorem
1]{Gallagher}) if choosing $\omega(q)=\{1\}$ and $\{q\}$
respectively.
\end{remark}

\subsection{Cross fibering principle}\label{section 2}

The cross fibering principle of Beresnevich, Haynes and Velani
(\cite[Theorem 3]{BHV}) provides an elegant way to verify  a set or
its complement in a product measure space of two $\sigma$-finite
ones is null. In the following we will develop a higher-dimensional
analogue. To help appreciate its proof, we first state and prove it
in its simplest form as follows:

\begin{theorem}\label{theorem 21}
Let $A\subset\mathbb{I}^n$ be a measurable set such that for any
line $L$ parallel to one of the coordinate axes, ${\mathcal
M}_{1}(A\cap L)=0$ or $1$. Then ${\mathcal M}_{n}(A)=0$ or $1$.
\end{theorem}

\begin{proof} We argue by induction. Since  there is nothing to do in the $n=1$ case,
 we may assume $n\geq2$ and Theorem \ref{theorem 21} is true for all dimensions less than $n$. For any $x\in\mathbb{I}$
and $y\in\mathbb{I}^{n-1}$, define as usual the sections of $A$
through $x$ and $y$ respectively by
\begin{align*}
A_x&=\{y\in \mathbb{I}^{n-1}:(x,y)\in A\},
\\A^y&=\{x\in\mathbb{I}:(x,y)\in A\}.
\end{align*}
In view of the induction hypothesis, ${\mathcal M}_{n-1}(A_x)=0$ or
$1$, ${\mathcal M}_{1}(A^y)=0$ or $1$. Define
\begin{align*}
X_1&=\{x\in\mathbb{I}:{\mathcal M}_{n-1}(A_x)=0\},\\
X_2&=\{x\in\mathbb{I}:{\mathcal M}_{n-1}(A_x)=1\},\\
Y_1&=\{y\in\mathbb{I}^{n-1}:{\mathcal M}_{1}(A^y)=0\},\\
Y_2&=\{y\in\mathbb{I}^{n-1}:{\mathcal M}_{1}(A^y)=1\}.
\end{align*}
By Fubini's theorem, ${\mathcal M}_n(A)={\mathcal
M}_1(X_2)={\mathcal M}_{n-1}(Y_2)$. Since $A=\cup_{i,j} (A\cap
(X_i\times Y_j))$,
\begin{align*} {\mathcal M}_n(A)&\leq\sum_{i,j}{\mathcal
M}_n(A\cap (X_i\times Y_j)) ={\mathcal M}_n(A\cap (X_2\times
Y_2))\\&\leq {\mathcal M}_1(X_2)\cdot{\mathcal
M}_{n-1}(Y_2)={\mathcal M}_n(A)^2,
\end{align*} which easily implies
${\mathcal M}_{n}(A)=0$ or $1$. This finishes the proof.
\end{proof}

The next theorem is a straightforward generalization of Theorem
\ref{theorem 21}, whose proof is pretty much the same as that of the
previous one with obvious modification.

\begin{theorem}\label{theorem 22}
View $\mathbb{I}^{nm}$ as the product of $m$ coordinate planes
$\mathbb{I}^n$. Let $A\subset\mathbb{I}^{nm}$ be a measurable set
such that for any $n$-dimensional plane $\Pi$ parallel to one of the
coordinate planes, ${\mathcal M}_{n}(A\cap \Pi)=0$ or $1$. Then
${\mathcal M}_{nm}(A)=0$ or $1$.
\end{theorem}

\subsection{${\mathcal A}$ and ${\mathcal B}$}

A relook at the definitions of ${\mathcal
A}_{n,m}(\Psi_1,\ldots,\Psi_m)$, ${\mathcal
B}_{n,m}(\Psi_1,\ldots,\Psi_m)$, ${\mathcal
A}_{n,m}^{\times}(\Psi)$, ${\mathcal B}_{n,m}^{\times}(\Psi)$ will
be useful to our later study. By letting
$\|\alpha\|\triangleq\min\{|\alpha+z|:z\in\mathbb{Z}\}$ for any
$\alpha\in\mathbb{R}$, and
\begin{align*}
\Theta(\mathbf{q},\mathbf{X})\triangleq\min\{|\mathbf{q}\mathbf{X}+z|:z\in\mathbb{Z}\
\mbox{with}\ z, \mathbf{q}\ \mbox{coprime}\}
\end{align*} for any
$\mathbf{q}\in\mathbb{Z}^n\backslash\{\mathbf{0}\}$ and any
$\mathbf{X}\in\mathbb{I}^n$, it is easy to see that ${\mathcal
A}_{n,m}(\Psi_1,\ldots,\Psi_m)$, ${\mathcal
B}_{n,m}(\Psi_1,\ldots,\Psi_m)$, ${\mathcal
A}_{n,m}^{\times}(\Psi)$, ${\mathcal B}_{n,m}^{\times}(\Psi)$ are
respectively the sets of $(\mathbf{X}^{(1)},\ldots,$
$\mathbf{X}^{(m)})\in\mathbb{I}^{nm}$ for which
\begin{align}
\|\mathbf{q}\mathbf{X}^{(j)}\|&<\Psi_j(\mathbf{q})\ \ \
(j=1,\ldots,m), \\
\Theta(\mathbf{q},\mathbf{X}^{(j)})&<\Psi_j(\mathbf{q})\ \ \
(j=1,\ldots,m), \\
\prod_{j=1}^m\|\mathbf{q}\mathbf{X}^{(j)}\|&<\Psi(\mathbf{q}),\\
\prod_{j=1}^m\Theta(\mathbf{q},\mathbf{X}^{(j)})&<\Psi(\mathbf{q})
\end{align}
for infinitely many $\mathbf{q}\in
\mathbb{Z}^n\backslash\{\mathbf{0}\}$.

\section{Zero-one laws in simultaneous Diophantine approximation}\label{n=1}

\begin{theorem}\label{theorem 31}
For any
$\{\Psi_j:\mathbb{Z}^n\backslash\{\mathbf{0}\}\rightarrow\mathbb{R}^{+}\}_{j=1}^m$,
${\mathcal M}_{nm}({\mathcal
A}_{n,m}(k_1\Psi_1,\ldots,k_m\Psi_m))\in\{0,1\}$ is independent of
$\{k_j>0\}_{j=1}^m$.
\end{theorem}

\begin{proof}
Obviously, by symmetry and (\ref{formula 12}), to prove Theorem
\ref{theorem 31} it suffices to show that $
 {\mathcal M}_{nm}\big({\mathcal
A}_{n,m}(k\Psi_1,\Psi_2,\ldots,\Psi_m)\big)\ \mbox{is independent
of}\ k>0.
$
 Note first
\[{\mathcal
A}_{n,m}(k\Psi_1,\Psi_2,\ldots,\Psi_m)=\bigcup_{(\mathbf{X}^{(2)},\ldots,\mathbf{X}^{(m)})\in\mathbb{I}^{n(m-1)}}{\mathcal
A}_{(\mathbf{X}^{(2)},\ldots,\mathbf{X}^{(m)})}(k\Psi_1,\Psi_2,\ldots,\Psi_m),\]
where ${\mathcal
A}_{(\mathbf{X}^{(2)},\ldots,\mathbf{X}^{(m)})}(k\Psi_1,\Psi_2,\ldots,\Psi_m)$
is the set of $\mathbf{X}^{(1)}\in\mathbb{I}^n$ for which
\[\|\mathbf{q}\mathbf{X}^{(1)}\|<k\Psi_1(\mathbf{q})\ \ \ \ \& \ \ \|\mathbf{q}\mathbf{X}^{(j)}\|<\Psi_j(\mathbf{q})\ \ \ \ (j=2,\ldots,m)\]
for infinitely many
$\mathbf{q}\in\mathbb{Z}^n\backslash\{\mathbf{0}\}$. Equivalently,
${\mathcal
A}_{(\mathbf{X}^{(2)},\ldots,\mathbf{X}^{(m)})}(k\Psi_1,\Psi_2,\ldots,\Psi_m)$
is the set of  $\mathbf{X}^{(1)}\in\mathbb{I}^n$  for which
$\|\mathbf{q}\mathbf{X}^{(1)}\|<k\Psi_1(\mathbf{q})\chi_{\mathbf{W}}(\mathbf{q})$
for infinitely many
$\mathbf{q}\in\mathbb{Z}^n\backslash\{\mathbf{0}\}$, where
$\mathbf{W}$ is the  set of solutions
$\mathbf{q}\in\mathbb{Z}^{n}\backslash\{\mathbf{0}\}$ to
$\|\mathbf{q}\mathbf{X}^{(j)}\|<\Psi_j(\mathbf{q})\ (j=2,\ldots,m)$.
By Lemma \ref{lemma 21} and Fubini's theorem,  ${\mathcal
M}_{nm}\big({\mathcal A}_{n,m}(k\Psi_1,\Psi_2,\ldots,\Psi_m)\big)$
is independent of $k>0$. This finishes the proof.
\end{proof}

\begin{theorem}\label{theorem 32}
For any
$\{\Psi_j:\mathbb{Z}^n\backslash\{\mathbf{0}\}\rightarrow\mathbb{R}^{+}\}_{j=1}^m$,
${\mathcal M}_{nm}({\mathcal
B}_{n,m}(k_1\Psi_1,\ldots,k_m\Psi_m))\in\{0,1\}$ is independent of
$\{k_j>0\}_{j=1}^m$.
\end{theorem}

\begin{proof}
The proof of Theorem \ref{theorem 32} is pretty much the same as
that of Theorem \ref{theorem 31} with obvious modification, and we
leave the details to the interested readers.
\end{proof}

\begin{theorem}\label{theorem 33}
For any
$\{\Psi_j:\mathbb{Z}^n\backslash\{\mathbf{0}\}\rightarrow\mathbb{R}^{+}\}_{j=1}^m$,
${\mathcal M}_{nm}({\mathcal
C}_{n,m}(k_1\Psi_1,\ldots,k_m\Psi_m))\in\{0,1\}$ is independent of
$\{k_j>0\}_{j=1}^m$.
\end{theorem}

\begin{proof}
Obviously, by symmetry and (\ref{formula 12}), to prove Theorem
\ref{theorem 33} it suffices to show that $
 {\mathcal M}_{nm}\big({\mathcal
C}_{n,m}(k\Psi_1,\Psi_2,\ldots,\Psi_m)\big)\ \mbox{is independent
of}\ k>0. $
 Note first
\[{\mathcal
C}_{n,m}(k\Psi_1,\Psi_2,\ldots,\Psi_m)=\bigcup_{(\mathbf{X}^{(2)},\ldots,\mathbf{X}^{(m)})\in\mathbb{I}^{n(m-1)}}{\mathcal
C}_{(\mathbf{X}^{(2)},\ldots,\mathbf{X}^{(m)})}(k\Psi_1,\Psi_2,\ldots,\Psi_m),\]
where ${\mathcal
C}_{(\mathbf{X}^{(2)},\ldots,\mathbf{X}^{(m)})}(k\Psi_1,\Psi_2,\ldots,\Psi_m)$
is the set of $\mathbf{X}^{(1)}\in\mathbb{I}^n$ for which
\[|\mathbf{q}\mathbf{X}^{(1)}+p_1|<k\Psi_1(\mathbf{q})\ \ \ \ \& \ \ |\mathbf{q}\mathbf{X}^{(j)}+p_j|<\Psi_j(\mathbf{q})\ \ \ \ (j=2,\ldots,m)\]
for infinitely many coprime pairs
$\mathbf{q}\in\mathbb{Z}^n\backslash\{\mathbf{0}\}$ and
$(p_1,\ldots,p_m)\in\mathbb{Z}^m$. For each non-empty solution set
$\mathbf{S}(\mathbf{q})=\mathbf{S}_{\mathbf{X}^{(2)},\ldots,\mathbf{X}^{(m)};\Psi_2,\ldots,\Psi_m}(\mathbf{q})$
of $(p_2,\ldots,p_m)\in\mathbb{Z}^{m-1}$ to
\[|\mathbf{q}\mathbf{X}^{(j)}+p_j|<\Psi_j(\mathbf{q})\ \ \
(j=2,\ldots,m),\] we define a subset of divisors of
$\mbox{gcd}(\mathbf{q})$ by $ \omega(\mathbf{q})=
\{\mbox{gcd}(\mathbf{q},\gamma):\gamma\in\mathbf{S}(\mathbf{q})\}.$
By letting $\mathbf{W}$ be the support of $\mathbf{S}$, it is easy
to see that
\[{\mathcal
C}_{(\mathbf{X}^{(2)},\ldots,\mathbf{X}^{(m)})}(k\Psi_1,\Psi_2,\ldots,\Psi_m)={\mathcal
H}_n(\omega,k\Psi_1\chi_{\mathbf{W}}).\] By Lemma \ref{lemma 21} and
Fubini's theorem, ${\mathcal M}_{nm}\big({\mathcal
C}_{n,m}(k\Psi_1,\Psi_2,\ldots,\Psi_m)\big)$ is independent of
$k>0$. This finishes the proof.
\end{proof}

\section{Zero-one laws in multiplicative Diophantine
approximation}\label{section 4}

\begin{theorem}\label{theorem 41}
For any
$\Psi:\mathbb{Z}^n\backslash\{\mathbf{0}\}\rightarrow\mathbb{R}^{+}$,
${\mathcal M}_{nm}({\mathcal C}_{n,m}^{\times}(k\Psi))\in\{0,1\}$ is
independent of $k>0$.
\end{theorem}

\begin{proof}
We first prove the independence property.  Note
\[{\mathcal
C}_{n,m}^{\times}(k\Psi)=\bigcup_{(\mathbf{X}^{(2)},\ldots,\mathbf{X}^{(m)})\in\mathbb{I}^{n(m-1)}}{\mathcal
C}^{\times}_{(\mathbf{X}^{(2)},\ldots,\mathbf{X}^{(m)})}(k\Psi),\]
where ${\mathcal
C}^{\times}_{(\mathbf{X}^{(2)},\ldots,\mathbf{X}^{(m)})}(k\Psi)$ is
the set of $\mathbf{X}^{(1)}\in\mathbb{I}^n$ for which
$\prod_{j=1}^m|\mathbf{q}\mathbf{X}^{(j)}+p_j|<k\Psi(\mathbf{q})$
holds for infinitely many coprime pairs
$\mathbf{q}\in\mathbb{Z}^n\backslash\{0\}$ and
$(p_1,\ldots,p_m)\in\mathbb{Z}^m$. For each
$\mathbf{q}\in\mathbb{Z}^n\backslash\{\mathbf{0}\}$, let
$\omega(\mathbf{q})$ be the set of all divisors of
$\mbox{gcd}(\mathbf{q})$, and define
\[M_{\mathbf{q}}(\gamma)=\min\big\{\prod_{j=2}^m|\mathbf{q}\mathbf{X}^{(j)}+p_j|:\mbox{gcd}(\mathbf{q},p_2,\ldots,p_m)=\gamma\big\}\
\  \ (\forall\gamma\in\omega(\mathbf{q})).\] It is easy to see that
${\mathcal
C}^{\times}_{(\mathbf{X}^{(2)},\ldots,\mathbf{X}^{(m)})}(k\Psi)$ is
the set of $\mathbf{X}^{(1)}\in\mathbb{I}^n$ for which
$|\mathbf{q}\mathbf{X}^{(1)}+p_1|\cdot
M_{\mathbf{q}}(\gamma)<k\Psi(\mathbf{q})$ holds for infinitely many
$(\mathbf{q},p_1)$ with $p_1$ coprime to some
$\gamma\in\omega(\mathbf{q})$. If there exists a pair
$(\mathbf{q},\gamma)$ with $M_{\mathbf{q}}(\gamma)=0$ and
$\Psi(\mathbf{q})>0$, then we obviously have ${\mathcal
C}^{\times}_{(\mathbf{X}^{(2)},\ldots,\mathbf{X}^{(m)})}(k\Psi)=\mathbb{I}^n$.
So by expelling such kind of existences and assuming
$\frac{0}{0}=0$,
\[{\mathcal
C}^{\times}_{(\mathbf{X}^{(2)},\ldots,\mathbf{X}^{(m)})}(k\Psi)={\mathcal
H}_n(\omega,\frac{k\Psi(\mathbf{q})}{M_{\mathbf{q}}(\gamma)}).\] By
Lemma \ref{lemma 21}, ${\mathcal M}_n({\mathcal
C}^{\times}_{(\mathbf{X}^{(2)},\ldots,\mathbf{X}^{(m)})}(k\Psi))\in\{0,1\}$
is independent of $k>0$. By Fubini's theorem, ${\mathcal
M}_{nm}({\mathcal C}_{n,m}^{\times}(k\Psi))$ is independent of
$k>0$.

Next we prove the zero-one property. To this purpose, by Theorem
\ref{theorem 22} it suffices to prove that, for example, for any
fixed $\mathbf{X}^{(2)},\ldots,\mathbf{X}^{(m)}\in\mathbb{I}^n$, the
set of $\mathbf{X}^{(1)}\in\mathbb{I}^n$ for which
$\prod_{j=1}^m|\mathbf{q}\mathbf{X}^{(j)}+p_j|<k\Psi(\mathbf{q})$
holds for infinitely many coprime pairs
$\mathbf{q}\in\mathbb{Z}^n\backslash\{\mathbf{0}\}$ and
$(p_1,\ldots,p_m)\in\mathbb{Z}^m$ has $n$-dimensional measure either
0 or 1. But this fact has been proved previously, we are done.
\end{proof}

\begin{remark}
Similar to the proof of the zero-one part of Theorem \ref{theorem
41}, one can replace (\ref{formula 12}) with Theorem \ref{theorem
22} as substitute tool in the proofs of the corresponding parts of
Theorems \ref{theorem 31}$\sim$\ref{theorem 33}.
\end{remark}

\begin{theorem}\label{theorem 43}
For ${\mathcal G}\in\{{\mathcal A},{\mathcal B}\}$ and any
$\Psi:\mathbb{Z}^n\backslash\{\mathbf{0}\}\rightarrow\mathbb{R}^{+}$,
${\mathcal M}_{nm}({\mathcal G}_{n,m}^{\times}(k\Psi))\in\{0,1\}$ is
independent of $k>0$.
\end{theorem}

\begin{proof}
The proof of Theorem \ref{theorem 43} is pretty much the same as
those of Theorem \ref{theorem 31} and Theorem \ref{theorem 41} with
obvious modification, and we leave the details to the interested
readers  (see also \cite{BHV}).
\end{proof}

As consequences of all the theorems established so far, we have
(\ref{formula 13}) as well as \begin{align}{\mathcal
M}_{nm}\big(\bigcup_{k\in\mathbb{N}}{\mathcal
F}_{n,m}^{\times}(k\Psi)\big)={\mathcal
M}_{nm}\big(\bigcap_{k\in\mathbb{N}}{\mathcal
F}_{n,m}^{\times}(\frac{\Psi}{k})\big).
\end{align}

\section{Further results and questions}

\subsection{Simultaneous
approximation} For the case $\Psi_1=\cdots=\Psi_m$ in simultaneous
Diophantine approximation, we refer to \cite{BBDV} for a survey of a
series of conjectures (note in particular the Duffin-Schaeffer
conjecture and the Catlin conjecture) and
\cite{BV2,Gallagher3,Haynes,PollingtonVaughan,Schmidt,Vaaler} for
several remarkable progresses. We also highlight the following
clear-cut theorems without monotonicity assumptions, due to
respectively Gallagher-Schmidt-Beresnevich-Velani (see e.g.
\cite{BV2}) and Pollington-Vaughan (\cite{PollingtonVaughan}).

\begin{theorem} Let $mn>1$ and
$\Psi(\mathbf{q})=\psi(|\mathbf{q}|):\mathbb{Z}^n\backslash\{\mathbf{0}\}\rightarrow\mathbb{R}^{+}$.
Then
\[{\mathcal M}_{nm}({\mathcal
A}_{n,m}(\Psi,\ldots,\Psi))=1\Leftrightarrow {\mathcal
M}_{nm}({\mathcal
C}_{n,m}(\Psi,\ldots,\Psi))=1\Leftrightarrow\sum_{q=1}^{\infty}q^{n-1}\psi(q)^m=\infty.\]
\end{theorem}

\begin{theorem}
 Let $m>1$ and
$\psi:\mathbb{Z}\backslash\{0\}\rightarrow\mathbb{R}^{+}$. Then
\[{\mathcal M}_{m}({\mathcal
B}_{1,m}(\psi,\ldots,\psi))=1\Leftrightarrow\sum_{q\in\mathbb{Z}\backslash\{0\}}\big(\frac{\psi(q)\varphi(|q|)}{|q|}\big)^m=\infty,\]
where $\varphi$ is Euler's totient function.
\end{theorem}

\subsection{Multiplicative
approximation}  For the case ${\mathcal G}\in\{{\mathcal A},
{\mathcal B}\}\ \&\ n=1$ in multiplicative Diophantine
approximation, we refer to \cite{BHV} for appropriate conjectures
and \cite{BHV,Gallagher2} for Duffin-Schaeffer and
Khintchine-Groshev types theorems. For the case ${\mathcal
G}={\mathcal C}\ \&\ n=1\ \&\ m\geq2$, we would like to propose the
following conjecture.

\begin{Conjecture} Let $m\geq2$ and
$\psi:\mathbb{Z}\backslash\{0\}\rightarrow[0,\frac{1}{2}]$. Then
\[{\mathcal M}_m({\mathcal
C}_{1,m}^{\times}(\psi))=1\Leftrightarrow\sum_{q\in\mathbb{Z}\backslash\{0\}}\psi(q)\log^{m-1}\frac{1}{\psi(q)}=\infty.\]
\end{Conjecture}

\subsection{Inhomogeneous
approximation} Given
$\{\Psi_j:\mathbb{Z}^n\backslash\{\mathbf{0}\}\rightarrow\mathbb{R}^{+}\}_{j=1}^m$,
 $\mathbf{b}=(b_1,\ldots,b_m)\in\mathbb{R}^m$,  let ${\mathcal
A}_{n,m}^{\mathbf{b}}(\Psi_1,\ldots,\Psi_m)$,  ${\mathcal
B}_{n,m}^{\mathbf{b}}(\Psi_1,\ldots,\Psi_m)$,  ${\mathcal
C}_{n,m}^{\mathbf{b}}(\Psi_1,\ldots,\Psi_m)$ denote the sets of
$(\mathbf{X}^{(1)},\ldots,\mathbf{X}^{(m)})\in\mathbb{I}^{nm}$ for
which \[ |\mathbf{q}\mathbf{X}^{(j)}+p_j+b_j|<\Psi_j(\mathbf{q})\ \
\ (j=1,\ldots,m) \] holds for infinitely many pairs $\mathbf{q}\in
\mathbb{Z}^n\backslash\{\mathbf{0}\}$ and
$\mathbf{p}=(p_1,\ldots,p_m)\in\mathbb{Z}^m$ subject to respectively
1) free condition on $\mathbf{q}$ and $\mathbf{p}$; 2) local
coprimality condition on $\mathbf{q}$ and $p_j$ for each $j$; 3)
global coprimality condition on $\mathbf{q}$ and $\mathbf{p}$. In
view of quite a few examples in \cite{B,BBV,BBDV,BVV,BV3}, it is
reasonable to raise the following question.

\begin{Question} Is it true that  ${\mathcal
M}_{nm}({\mathcal F}_{n,m}^{\mathbf{b}}(\Psi_1,\ldots,\Psi_m))$ is
independent of $\mathbf{b}\in\mathbb{R}^m$?
\end{Question}

In much the same way, one can first introduce  ${\mathcal
A}_{n,m}^{\mathbf{b},\times}(\Psi)$, ${\mathcal
B}_{n,m}^{\mathbf{b},\times}(\Psi)$ and ${\mathcal
C}_{n,m}^{\mathbf{b},\times}(\Psi)$, then propose a similar
question.

\begin{Question} Is it true that  ${\mathcal
M}_{nm}({\mathcal F}_{n,m}^{\mathbf{b},\times}(\Psi))$ is
independent of $\mathbf{b}\in\mathbb{R}^m$?
\end{Question}


\begin{thebibliography}{99}

\bibitem{B}
D. Badziahin, \emph{Inhomogeneous Diophantine approximation on
curves and Hausdorff dimension}, Adv. Math. 223 (2010), 329--351.

\bibitem{BBV}
D. Badziahin, V. Beresnevich, S. Velani, \emph{Inhomogeneous theory
of dual Diophantine approximation on manifolds}, arXiv:1009.5638,
accepted by Adv. Math..

\bibitem{BBDV}
V. Beresnevich, V. Bernik, M. Dodson, S. Velani, \emph{Classical
metric Diophantine approximation revisited}, in Analytic Number
Theory Essays in Honour of Klaus Roth, edited by W.~W.~L. Chen et
al., Cambridge Univ. Press, Cambridge, (2009), 38--61.

\bibitem{BHV}
V. Beresnevich, A. Haynes, S. Velani, \emph{Multiplicative zero-one
laws and metric number theory}, arXiv:1012.0675, accepted by Acta
Arith..


\bibitem{BVV}
V.~V. Beresnevich,  R.~C. Vaughan, S.~L. Velani, \emph{Inhomogeneous
Diophantine approximation on planar curves}, Math. Ann. 349 (2011),
929--942.

\bibitem{BV1}
V. Beresnevich, S. Velani, \emph{A note on zero-one laws in metrical
Diophantine approximation}, Acta Arith. 133 (2008), 363--374.




\bibitem{BV2}
V. Beresnevich, S. Velani, \emph{Classical metric Diophantine
approximation revisited: the Khintchine-Groshev theorem}, Int. Math.
Res. Not. 2010 (2010), 69--89.

\bibitem{BV3}
V. Beresnevich, S. Velani, \emph{An inhomogeneous transference
principle and Diophantine approximation}, Proc. London Math. Soc.
101 (2010), 821--851.




\bibitem{Cassels}
J.~W.~S. Cassels, \emph{Some metrical theorems in Diophantine
approximation. I}, Proc. Cambridge Philos. Soc. 46 (1950), 209--218.

\bibitem{DuffinSchaeffer}
R.~J. Duffin, A.~C. Schaeffer,\emph{ Khintchine's problem in metric
Diophantine approximation}, Duke Math. J. 8 (1941), 243--255.


\bibitem{Gallagher}
P. Gallagher, \emph{Approximation by reduced fractions}, J. Math.
Soc. Japan 13 (1961), 342--345.

\bibitem{Gallagher2}
P. Gallagher, \emph{Metric simultaneous Diophantine approximation},
J. London Math. Soc.  37 (1962), 387--390.

\bibitem{Gallagher3}
P.~X. Gallagher, \emph{Metric simultaneous Diophantine approximation
(II)}, Mathematika  12 (1965), 123--127.



\bibitem{Harman}
G. Harman, Metric Number Theory, Clarendon Press, Oxford, 1998.

\bibitem{Haynes}
A.~K. Haynes, A.~D. Pollington, S.~L. Velani, \emph{The
Duffin-Schaeffer Conjecture with extra divergence}, arXiv:0811.1234,
to appear in Math. Ann..





\bibitem{PollingtonVaughan}
A.~D. Pollington, R.~C. Vaughan, \emph{The $k$-dimensional Duffin
and Schaeffer conjecture}, Journal de Th\'{e}orie des Nombres de
Bordeaux 1 (1989), 81--88.


\bibitem{Roth}
K.~F. Roth, \emph{Rational approximations to algebraic numbers},
Mathematika  2 (1955), 1--20.


\bibitem{Schmidt}
W.~M. Schmidt, \emph{A metrical theorem in Diophantine
approximation}, Canadian J. Math. 12 (1960), 619--631.

\bibitem{Sprindzuk}
V. Sprind$\breve{z}$uk, Metric Theory of Diophantine Approximation,
John Wiley \& Sons, New York, 1979. (English translation)

\bibitem{Vilchinski}
V.~T. Vilchinski, \emph{On simultaneous approximations by
irreducible fractions}, Vestsi Akad. Navuk BSSR Ser. Fiz.-Mat. Navuk
140 (1981), 41--47.  (In Russian)

\bibitem{Vaaler}
J.~D. Vaaler, \emph{On the metric theory of Diophantine
approximation}, Pacific J. Math. 76 (1978), 527--539.


\end{thebibliography}
\end{document}